\documentclass[11pt,letter]{article}
\pdfoutput=1
\usepackage[paperwidth=8.5in,left=1.0in,right=1.0in,top=1.0in,bottom=1.0in,
  paperheight=11.0in]{geometry}
\usepackage{cite}
\usepackage{color}
\usepackage{soul}
\usepackage{graphicx}
\usepackage{amsthm}
\usepackage{algorithm}
\usepackage[noend]{algorithmic}
\usepackage{hyperref}
\usepackage{chngcntr}
\usepackage{epstopdf}
\usepackage{amsmath}
\usepackage{amssymb}

\usepackage{caption}
\usepackage{subcaption}

\usepackage{comment}

\makeatletter
\def\@copyrightspace{\relax}
\makeatother

\newtheorem{theorem}{Theorem}
\newtheorem{lemma}{Lemma}
\newtheorem{proposition}{Proposition}

\newtheorem{claim}{Claim}

\title{Planar Graphs that Need Four Pages\footnote{To appear in Journal of Combinatorial Theory, Series B (submitted June 12, 2019).}}
\author{Mihalis Yannakakis\\
       {Computer Science}\\
       {Columbia University}\\
       {New York, NY}\\
       {mihalis@cs.columbia.edu}
}
\date{}
\begin{document}
\setlength{\textfloatsep}{10 pt}
\clearpage\maketitle
\thispagestyle{empty}
\begin{abstract}
We show that there are planar graphs that require four pages in any book embedding.
\end{abstract}

\section{Introduction}

A {\em book embedding} of a graph consists of a linear order of its nodes and a partitioning of its edges, so that the nodes can be placed in
order on a straight line (the ``spine" of the book) and the edges 
in each part can be drawn on a separate half-plane bordered by the line 
(a ``page" of the book) so that the edges on the same page do not intersect.
The objective is to find a book embedding
that uses the minimum number of pages.
This minimum number is called
the {\em pagenumber} (or {\em book thickness}) of the graph.

Book embeddings were introduced in \cite{Oll} and \cite{BK}.
They were studied in connection with an approach to fault-tolerant VLSI design
\cite{R,CLR}, and have applications also in sorting
using stacks, in graph drawing, complexity theory, and other areas.
Computationally, the problem of computing the pagenumber of a graph is hard:
it is NP-complete to decide whether a planar graph has pagenumber 2 \cite{W,CLR}.
Note that in the case of two pages, 
the crux of the problem is the node-embedding part
(the linear ordering of the nodes): once the node ordering is fixed, it is easy to test whether two pages suffice. In general, the subproblem of minimizing the number of pages for a given fixed node ordering is itself also NP-hard \cite{GJMP}.

Graphs with pagenumber one are exactly the outerplanar graphs.
Graphs with pagenumber two are exactly the planar subhamiltonian graphs, i.e.
the subgraphs of planar Hamiltonian graphs \cite{BK}.
In \cite{Ya1,Ya2} we showed that all planar graphs can be embedded in four pages,
and gave a linear time algorithm for this purpose
(references \cite{BS} and \cite{H} gave
earlier algorithms that embed planar graphs in nine and seven pages respectively).
In the conference paper \cite{Ya1} we stated also that there 
are planar graphs that require four pages,
and outlined briefly the approach and the structure of the construction.
The present paper gives the full details of the construction 
and the proof that the constructed planar graph has pagenumber at least four.

Besides planar graphs, there has been extensive work 
on book embeddings for various other classes of graphs, for example graphs of bounded genus \cite{HI,M}, bounded treewidth \cite{GH}, 1-planar graphs \cite{BBKR}.

Before getting into the technical details of the construction and the proof,
we make a few remarks regarding the issues involved and our approach to address them.
To show a lower bound of four on the pagenumber of planar graphs
means finding a planar graph $G$ and showing that 
no matter how we order its nodes and 
how we partition its edges into three pages, there will be a violation
(two conflicting edges on the same page).
A major obstacle in this regard stems from the computational complexity of the
problem: The problem is NP-complete, which means that if $NP \neq coNP$,
as is widely believed to be the case, there is in general no short way
to prove the nonexistence of a suitable node ordering 
and edge partitioning for a given graph;
that is, any proof has to be in general at least superpolynomially long
in the size of the graph, and will likely amount  to 
examining essentially all (or at least a large number of) possibilities 
to ensure that they do not work.
An important difference here is that we do not have to deal with an arbitrary
general graph, but with a specific graph $G$ of our own design
that is suitable for our purpose.
If there was a very small suitable graph $G$, then the complexity would not be
an issue. This is the case for example in the graph coloring problem,
where there is an extremely small planar graph (namely, $K_4$) that needs 4 colors,
so the lower bound is trivial.
Unfortunately, this does not seem to be the case in the book embedding problem.
For example, \cite{BKZ} studied experimentally the book embedding problem
by formulating it in terms of the SAT (Boolean formula satisfiability) problem and using
a SAT solver; the software handles graphs with up to 500-700 nodes. In their experiments 
searching for planar graphs that require four pages 
(they tried both random graphs and crafted graphs
in certain families) none was encountered, which led the authors to
hypothesize that perhaps three pages are enough for all planar graphs.

The way we deal with the complexity of the problem in our construction
is by building up the graph (and the proof) in stages, to control
the exponential explosion in case analysis. As in
NP-completeness reductions, we design and use gadgets (`small' graphs that have
useful properties) as building blocks.   We start with a small graph $Q_1$ (10 nodes)
which we analyze in some detail to characterize its book embeddings under some
strong restrictions. Then we use this gadget to analyze a somewhat larger graph $Q_2$,
and this is used in turn for a larger gadget $Q$ under weaker restrictions. 
The gadgets are used in the construction of the final graph $G$ that cannot be embedded
in three pages. The properties we showed for the gadgets restrict significantly the
possible book embeddings of $G$ and make the analysis tractable.

The rest of the paper is organized as follows.
Section 2 provides basic definitions and some simple observations.
In Section 3 we design the gadgets and prove their properties.
Section 4 gives the definition of the graph $G$ and proves that
$G$ requires four pages.

\section{Preliminaries}

Let $G=(N,E)$ be a (undirected) graph, and $\pi$ a linear ordering
(a permutation) of its nodes.
We say that two edges $(a,b), (c,d)$ {\em conflict} in the ordering $\pi$
if $\pi(a) < \pi(c) < \pi(b) < \pi(d)$ or $\pi(c) < \pi(a) < \pi(d) < \pi(b)$;
that is, if we place the nodes on a line ordered according to $\pi$
and draw the edges (as curves) on a half-plane bordered by the line,
two edges conflict iff they intersect.
A {\em book embedding} of $G$ in $k$ pages consists of 
(1) a linear ordering $\pi$ of its nodes, and (2) a coloring of its edges
with $k$ colors (the ``pages") so that conflicting edges in $\pi$ receive different colors.
The {\em pagenumber} of $G$ is the minimum number $k$ of pages such that
$G$ has a book embedding in $k$ pages.

We will usually show in figures a (partial) book embedding in 3 pages by showing
the positions of the nodes on a line and showing edges of color 1 as red dashed curves, color 2 as solid blue, and color 3 as dotted green.
To simplify notation, we will identify the positions of the nodes 
on the line with the nodes, and refer for example to node $a$ on the line
(instead of point $\pi(a)$), to interval $(a,b)$ of the line
(instead of interval $(\pi(a),\pi(b))$), and so forth.

Alternatively, a book embedding of a graph $G$ can be defined
by a mapping of its nodes to (distinct) points on a circle,
drawing the edges as straight line segments (chords) inside the circle
and coloring them with $k$ colors so that intersecting edges receive
different colors.
From such a circle embedding  one can obtain a linear embedding
by cutting the circle at any point and then ordering the nodes on the line
in either of the two directions, clockwise or counterclockwise
(the edges retain their colors).
Conversely, given a linear embedding, one
can obtain a circle embedding by connecting the two ends of the line
to form a circle that encloses all the edges.
Note that a circle embedding corresponds to $2|N|$ linear embeddings
that are essentially equivalent.

Given a book embedding of a graph $G$, we will say that an edge {\em exits}
an interval $(a,b)$ if one node of the edge is in the open interval $(a,b)$ and the other node is outside the closed interval $[a,b]$.
Similarly, for a circle embedding, an edge exits an arc $(a,b)$ if 
the two nodes of the edge are on the two different open arcs 
of the circle between $a$ and $b$.

The following proposition gives some basic
simple observations that are used throughout the paper, usually without making explicit reference to the proposition.

\begin{proposition}\label{closed}
(1) If in a book embedding of a graph $G$ there is a path $p_i$ between
two nodes $a, b$ all of whose edges have the same color $i$,
then there is no edge of color $i$ that exits the interval $(a,b)$ 
and connects two nodes that are not on the path $p_i$.\\
(2) If in a 3-page embedding of $G$ there are paths $p_1, p_2, p_3$ of all 3 colors $1, 2, 3$ respectively between two nodes $a,b$, 
then there is no edge that exits the interval $(a,b)$ and connects two nodes that are not in any of these paths.\\ 
(3) If in a 3-page embedding of $G$ there are paths $p_1, p_2, p_3$ of all 3 colors $1, 2, 3$ between two nodes $a,b$,
then for every connected component $C$ of the graph $G \setminus (p_1 \cup p_2 \cup p_3)$ obtained by deleting the nodes of $p_1, p_2, p_3$ from $G$,
either all the nodes of $C$ are in the interval $(a,b)$ or they are
all outside $(a,b)$.
\end{proposition}
\begin{proof}
(1) Let $p_i$ be the path $a=v_0, v_1, \ldots, v_m=b$,
and let $(x,y)$ be a color-$i$ edge whose nodes are not on $p_i$.
If a node $v_j$ of the path is in the interval $(x,y)$, then the
next node $v_{j+1}$ must be also in $(x,y)$ because the edges
$(v_j,v_{j+1}), (x,y)$ have the same color, hence they do not conflict.
Thus, by induction, if $a=v_0$ is in the interval $(x,y)$ then
all the nodes of the path are in $(x,y)$, hence also $b$.
Similarly, if $a=v_0$ is not in the interval $(x,y)$
then $b$ is not in the interval $(x,y)$ either.
In either case, $(x,y)$ does not exit the interval $(a,b)$.

(2) Follows from (1).

(3) Any two nodes $x, y$ of $C$ are connected by a path 
$x=u_0, u_1, \ldots, u_l=y$ that does not contain any node of the paths $p_i$.
By (2), $u_j$ is in the interval $(a,b)$ iff $u_{j+1}$ is also in $(a,b)$.
Hence by induction, $x=u_0$ is in $(a,b)$ iff every node of the path,
and in particular $y$, is in $(a,b)$. 
\end{proof}

We will say that a node $u$ {\em reaches} an interval $(a,b)$ 
(or any subset of the line)
if there is an edge from $u$ to a node in the interval $(a,b)$
(resp., in the subset).
Thus, for example by the above Proposition (part 2),
if in a 3-page embedding there are paths $p_1, p_2, p_3$ of all three colors between nodes $a, b$,
then the nodes embedded in the interval $(a,b)$
cannot reach any nodes outside $(a,b)$ except possibly the nodes
of the paths $p_i$.

\section{The Gadgets}

Figure \ref{fig:q1} shows our first ``gadget" $Q_1$.
We will refer to nodes 1, 2 as the {\em outer terminals}, to nodes $a,b$ as the
{\em inner terminals}, and to nodes $c_1, c_2$ as the {\em centers}.
We will use $Q_1+12$ to denote the graph $Q_1$ with the additional edge $(1,2)$
connecting the outer terminals, and we will use $Q_1 +ab$ to denote the graph with the
additional edge $(a,b)$ connecting the inner terminals.

\begin{figure}[h]
\centering
\vspace*{-1.2cm}
\includegraphics[scale=0.6]{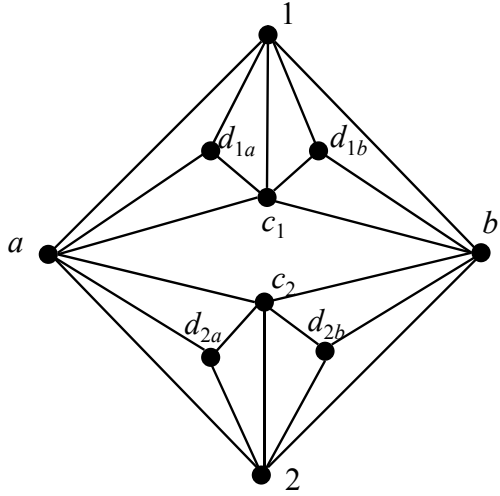}
\vspace*{-4cm}
\caption{The graph $Q_1$.}
\vspace*{-0.2cm}
\label{fig:q1}
\end{figure}

\begin{lemma}\label{lem:inside}
There is no 3-page embedding of $Q_1+12$ or of $Q_1+ab$ such that
(1) the inner terminals $a, b$ lie in the interval $(1,2)$ between the outer terminals, (2) the centers $c_1, c_2$ of $Q_1$ lie in the interval $(a,b)$, and (3) all edges from node 1 to the closed interval $[a,b]$
use the same color, say color 1, and all edges from node 2 to $[a,b]$
use color 2.
\end{lemma}
\begin{proof}
The proof is by contradiction.
Consider any 3-page embedding of $Q_1$ that satisfies conditions 1-3 of the lemma.
Suppose without loss of generality that the nodes $a,b, c_1, c_2$ are laid out 
in the order shown in Figure \ref{fig:inside}; the other possible embeddings with the order of $a,b$ and/or $c_1, c_2$ reversed, are symmetric. The outer terminals 1,2 are not shown explicitly in the figure: 
one of them, node $s \in \{1,2\}$ is at the left end of the line and the other outer terminal $t$ is to the right of $b$,
but we do not specify which is which so that we do not have to
distinguish cases; it is irrelevant for the following arguments which one of 1,2 is $s$ at the left end and which is $t$ on the right.
The edges to 1, 2 are shown in the figure as `dangling' segments with only one node.
Thus, the red dashed dangling edges have color 1 and go to node 1,
and the solid blue dangling edges have color 2 and go to terminal 2.
We shall prove that edge $(a,c_2)$ has color 2 and $(b,c_1)$ has color 1, as shown in the figure. 

\begin{figure}[h]
\centering
\vspace*{-1cm}
\includegraphics[scale=0.6]{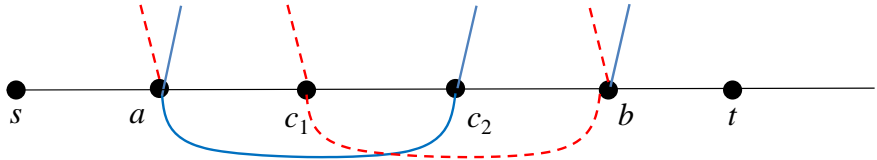}
\vspace*{-8.7cm}
\caption{Centers in $(a,b)$.}
\vspace*{-0.2cm}
\label{fig:inside}
\end{figure}

Edge $(a,c_2)$ conflicts with the edge $(1,c_1)$ which is colored 1 by the hypothesis, hence $(a,c_2)$ is colored 2 or 3.
Similarly edge $(b,c_1)$ conflicts with $(2,c_2)$, hence $(b,c_1)$ is
colored 1 or 3. The two edges $(a,c_2)$, $(b,c_1)$ conflict with each other, so they cannot both be colored 3. 
Suppose without loss of generality that $(a,c_2)$ is not colored 3, hence it is colored 2. 

Suppose that edge $(b,c_1)$ has color 3. We shall argue that it is
impossible then to embed legally node $d_{2a}$. Observe that node $c_2$ cannot reach any node, other than 1,2, outside the interval $[a,b]$ because of the color-3 edge $(b,c_1)$, the color-1 path $(a,1,b)$ and the color-2 path $(a,2,b)$. Hence $d_{2a}$ must lie in the interval $(a,b)$.
Since the edges from node 2 to this interval are colored 2 and $(a,c_2)$ is also colored 2, node $d_{2a}$ must lie in the interval $(c_2,b)$.
Then there is no legal color available for the edge $(a,d_{2a})$ because it conflicts with the color-1 edge $(c_1,1)$, the color-2 edge $(c_2,2)$ and the color-3 edge $(b,c_1)$.
We conclude that the edge $(b,c_1)$ does not have color 3, hence it
has color 1.

Consider now the possible positions of nodes $d_{1b}$ and $d_{2a}$.
We show first that they must be outside the interval $(a,b)$.
Suppose that $d_{1b}$ is inside the interval $(a,b)$,
to derive a contradiction.
Since $d_{1b}$ is adjacent to node 1 and edge $(c_1,b)$ is colored 1,
node  $d_{1b}$ cannot be in the interval $(c_1,b)$, hence it must be in
$(a,c_1)$. The edge $(b,d_{1b})$ must be colored 3 (since it conflicts
with the color-1 edge $(1,c_1)$ and the color-2 edge $(2,c_2)$).
Hence, node $c_2$ cannot reach any node, other than 1, 2, outside the
interval $(a,b)$ because of the color-1 path $(a,1,b)$, the color-2 path $(a,2,b)$, and the color-3 edge $(b,d_{1b})$. Therefore, node $d_{2a}$ 
must then also be inside the interval $(a,b)$.
However, node 2 cannot reach the subinterval $(a,c_2)$ because of the
color-2 edge $(a,c_2)$, and node $a$ cannot reach the subinterval
$(c_2,b)$ because of 
the color-1 edge $(c_1,b)$, the color-2 path $(c_2,2, b)$ and the
color-3 edge $(b,d_{1b})$
Thus, there is no legal position for node $d_{2a}$, contradiction.
We conclude that $d_{1b}$ must be outside the interval $(a,b)$.
By a symmetric argument, $d_{2a}$ also lies outside the interval $(a,b)$.

The edges $(c_1,d_{1b})$ and $(c_2,d_{2a})$ must be colored 3,
because they exit the interval $(a,b)$ and there are color-1 and -2 paths $(a,1,b)$ and $(a,2,b)$ connecting $a$ and $b$.

We can prove the lemma now for $Q_1+ab$:
The edge $(a,b)$ connecting the inner terminals
cannot be colored legally because it intersects the color-1 edge $(1,c_1)$, the
color-2 edge $(2,c_2)$ and the color-3 edge $(c_1,d_{1b})$.
This proves the claim for the graph $Q_1+ab$.

It remains to prove the claim for $Q_1+12$.
Since the edges $(c_1,d_{1b})$ and $(c_2,d_{2a})$ have the same color,
either $d_{1b}$ is to the left of $c_1$ and hence of $a$, or $d_{2a}$ is to the right of $c_2$ and hence of $b$ (or both).
Suppose without loss of generality that $d_{2a}$ is to the right of $b$. 
If it is left of $t$, i.e., within the same arc $(1,2)$ in the corresponding
embedding on a circle, then the edge $(a,d_{2a})$ must be also colored 3 (because it conflicts with the edges $(1,b), (2,b)$), but then
either $(a,d_{2a})$ or $(c_2,d_{2a})$ conflicts with the edge
$(c_1,d_{1b})$, a contradiction.
We conclude that $d_{2a}$ is right of $t$ , i.e.,
lies in the opposite arc $(1,2)$.
By a symmetric argument, the same holds for the node $d_{1b}$.
 
Since $(c_1,d_{1b})$ and $(c_2,d_{2a})$ have the same color (3),
hence do not intersect, nodes $s, a, c_1, c_2, b, t, d_{2a}, d_{1b}$
appear in this order. Consider the three edges $(d_{1b},b), (d_{2a},a), (s,t)$.
They all intersect a color-3 edge, namely  $(c_2,d_{2a}), (c_1,d_{1b}), (c_1,d_{1b})$ respectively,
hence they can only use the colors 1, 2.
However, the three edges $(d_{1b},b), (d_{2a},a), (s,t)$
conflict with each other, hence they cannot all be colored with two colors.
This proves the lemma for $Q_1+12$.
\end{proof}

\begin{lemma}\label{lem:outside}
It is not possible to embed $Q_1$ in three pages, such that
(1) the inner terminals $a, b$ lie in a subinterval $(u,v)$ between the outer terminals 1,2, (2) both centers $c_1, c_2$ lie outside the interval $[a,b]$, 
(3) all edges from node 1 to the interval $(u,v)$
use the same color, say color 1, and all edges from node 2 to $(u,v)$
use color 2, and (4)
all other edges exiting the interval $(u,v)$ (not connecting
to nodes 1, 2) are colored 3.
\end{lemma}
\begin{proof}
The proof is by contradiction.
Assume an embedding as in the lemma.
If $c_1$ is outside $(u,v)$ then both edges $(c_1,a)$, $(c_1,b)$ must have color 3
(by condition (4)), and similarly for $c_2$.
We cannot have both $c_1, c_2$ outside the interval $(u,v)$ because
then (at least) one of the edges $(c_1,a)$, $(c_1,b)$ would intersect
one of the edges $(c_2,a)$, $(c_2,b)$.
Therefore, at least one of $c_1, c_2$ must be in the interval $(u,v)$.

Assume without loss of generality that $c_1$ is in the
interval $(u,v)$, and, since
it is not in the interval $[a,b]$, assume wlog that
it is left of $a$, i.e., that it is in the subinterval $(u,a)$.
The edge $(c_1,b)$ must be colored 3 because of the edges $(1,a), (2,a)$. Node $a$ cannot have any edge exiting the interval $(u,b)$
to a node other than 1,2 because of condition (4), the color-3 edge $(c_1,b)$
and the color-1 and -2 edges $(1,b), (2,b)$. 
Therefore, $c_2$ lies in the interval $(u,b)$,
and since it is not in $[a,b]$, it is in $(u,a)$.
Assume without of loss of generality that $c_1$ is closer to $u$ than $c_2$;
see Figure \ref{fig:outside-same}.

\begin{figure}[h]
\centering
\vspace*{-1.1cm}
\includegraphics[scale=0.6]{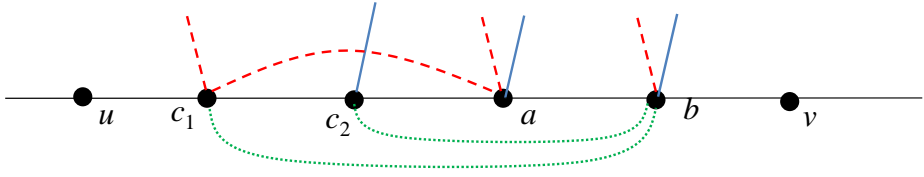}
\vspace*{-8.9cm}
\caption{Centers outside $(a,b)$.}
\vspace*{-0.2cm}
\label{fig:outside-same}
\end{figure}

The edge $(c_2,b)$ is colored 3, because of the edges
$(1,a), (2,a)$. This implies that the edge $(c_1,a)$ must be colored 1
because of the color-2 edge $(2,c_2)$.
Consider the possible position of node $d_{1a}$. Node $a$ cannot exit the interval $(c_1,b)$ because of the color-1 path $(c_1,1,b)$,
the color-2 path $(c_2,2,b)$, and the color-3 edge $(c_1,b)$.
On the other hand, node 1 cannot reach the interval $(c_1,a)$
because of the color-1 edge $(c_1,a)$. Furthermore,
$c_1$ cannot reach the interval $(a,b)$ because of the color-1 path
$(a,1,b)$, the color-2 path $(a,2,b)$ and the color-3 edge $(c_2,b)$
(see Figure \ref{fig:outside-same}).
Therefore, there is no legal position for node $d_{1a}$, a contradiction.
\end{proof}

Let $Q_2$ be the graph shown in Figure \ref{fig:q2} where each internal face (triangle)
is stellated twice more. That is, inside each triangle (e.g. $(1,a,d_{1a}), (1,d_{1a},c_1)$ etc.) we insert a center node with edges to the nodes of the triangle, and repeat 
this once more for each resulting triangle; these additional
nodes are not shown in the Figure so that it will not become too cluttered.
Note that $Q_2$ contains many copies of $Q_1$.
For example there is one copy with outer terminals 1,2 and inner terminals $a, b$.
Another copy of $Q_1$ has outer terminals $1, c_2$ and inner terminals $a, c_1$;
another has outer terminals $d_{1a}, c_2$ and inner terminals $a, c_1$.
And so forth.
We will use again $Q_2 +12$ to denote the graph $Q_2$ with the additional edge $(1,2)$.

\begin{figure}[h]
\centering
\vspace*{-1.1cm}
\includegraphics[scale=0.6]{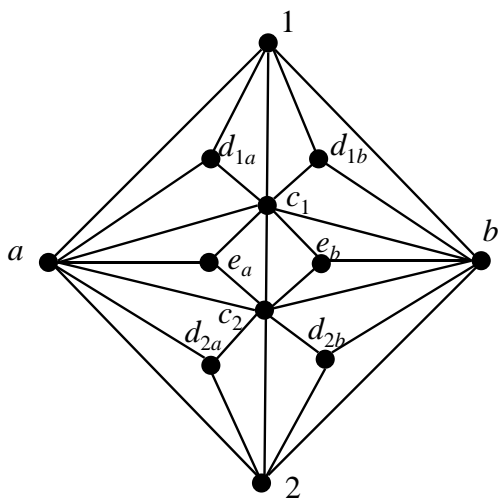}
\vspace*{-4.3cm}
\caption{The graph $Q_2$ (the triangles are stellated twice more)}
\vspace*{-0.3cm}
\label{fig:q2}
\end{figure}

\begin{lemma}\label{lem:in-out}
Suppose that in a 3-page embedding of $Q_2+12$, 
(1) the inner terminals $a, b$ lie in a subinterval $(u,v)$ between the outer terminals 1,2, (2) all edges from node 1 to the  interval $(u,v)$
use the same color, say color 1, and all edges from node 2 to $(u,v)$
use color 2, and (3) all other edges exiting the interval $(u,v)$ (not connecting
to nodes 1, 2) are colored 3.
Then one of the center nodes $c_1, c_2$ is inside the interval $(a,b)$
and the other one is outside $(a,b)$ but inside $(u,v)$. 
Furthermore, the embedding is as shown in Figure \ref{fig:in-out}
(up to reversing the order, switching $a,b$ and/or switching the indices $1,2$). 
That is, if the center node $c_1$ is outside
$(a,b)$ and it is in the interval $(u,a)$, then $d_{1b}$ is in the interval
$(b,v)$ and the edges $(c_1,c_2), (c_1,b), (c_1,d_{1b}), (a,c_2)$ all
have color 3.
\end{lemma}

\begin{figure}[h]
\centering
\vspace*{-0.8cm}
\includegraphics[scale=0.6]{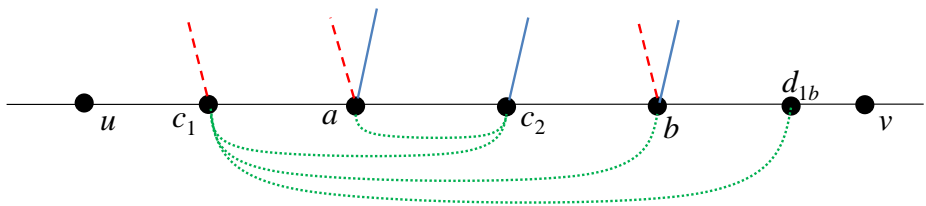}
\vspace*{-8.9cm}
\caption{Centers in and out.}
\vspace*{-0.2cm}
\label{fig:in-out}
\end{figure}

\begin{proof}
By Lemmas \ref{lem:inside}, \ref{lem:outside},
one of the centers $c_1, c_2$ must lie outside the interval $(a,b)$ and
one inside. Assume without loss of generality that $c_1$ is
outside $(a,b)$ and $c_2$ is inside.
The edge $(c_1,c_2)$ must have color 3 because of the
color-1 path $(a,1,b)$ and the color-2 path $(a,2,b)$.

We will show that the embedding conforms to the more detailed Figure \ref{fig:in-out2}
(or a symmetric one obtained by reversing the order, switching $a,b$ and/or switching $1,2$).
We will do this in several steps.
First we will show that $c_1$ is in the interval $(u,v)$.
Second we will show that $e_b$ and $d_{2b}$ have the indicated positions.
Third, we will identify the positions of $e_a$ and $d_{1a}$.
Finally, we will identify the position of $d_{1b}$ to complete the proof.

\begin{figure}[h]
\centering
\vspace*{-1cm}
\includegraphics[scale=0.6]{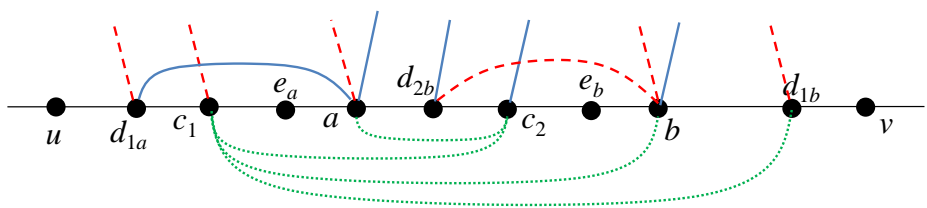}
\vspace*{-8.9cm}
\caption{Centers in and out: Detailed embedding.}
\vspace*{-0.2cm}
\label{fig:in-out2}
\end{figure}

\begin{claim}
Node $c_1$ is in the interval $(u,v)$.
\end{claim}
\begin{proof}
Suppose that $c_1$ is outside $(u,v)$. 
Then both edges $(c_1,a), (c_1,b)$ have color 3
by condition (3).
Since $a$ and $b$ are connected by paths $(a,1,b)$, $(a,2,b)$,
$(a,c_1,b)$ of all three colors, no edge can exit the
interval $(a,b)$ to any node other than 1, 2, $c_1$.
Since $c_2$ is inside the interval $(a,b)$, it follows that
all the nodes of $Q_2$ in the interior of the graph bounded
by the cycle $(a,c_1, b, 2)$ (which is connected) must be inside the interval $(a,b)$. Note that the subgraph bounded
by the cycle $(a, c_1, c_2, 2)$ contains $Q_1$ with $c_1, 2$ as the outer terminals and $a, c_2$ as the inner terminals,
which are adjacent.
The inner terminals $a, c_2$ lie in the same arc $(c_1, 2)$
(in a cyclic ordering),
all edges from $c_1$ to the interval $[a,c_2]$ must be colored 3
and all edges from 2 to $[a,c_2]$ must be colored 2.
By Lemma \ref{lem:inside}, at least one of the
two centers $e_a, d_{2a}$ must be outside the interval $[a,c_2]$.
Hence it must be in the interval $(c_2,b)$, and the edge connecting it
to node $a$ must be colored 1 because it conflicts with the edges
$(c_1,c_2)$ and $(2,c_2)$.
Similarly, the subgraph of $Q_2$ bounded by the cycle
$(b,c_1, c_2, 2)$ contains $Q_1$ with $c_1, 2$ as the outer terminals and $b, c_2$ as the inner terminals,
which are adjacent. By a symmetric argument,
at least one of the two centers $e_b, d_{2b}$ must be in the
interval $(a,c_2)$, and the edge that connects it to $b$ is also colored 1. Thus, there are two color-1 edges that intersect, a contradiction.
We conclude that $c_1$ is in the interval $(u,v)$.
\end{proof}

Since $c_1$ is outside the interval $(a,b)$ (see the beginning of the proof) and is inside the interval $(u,v)$ (by Claim 1), it follows
that $c_1$ is either in the interval $(u,a)$ or in $(b,v)$.
Assume without loss of generality that $c_1$ is in the interval $(u,a)$ 
as in Fig. \ref{fig:in-out}.

\begin{claim}
Node $e_b$ is in the interval $(c_2,b)$ and node $d_{2b}$ is in the interval $(a,c_2)$.
Edge $(b,d_{2b})$ has color 1 and edge $(a,c_2)$ has color 3.
\end{claim}
\begin{proof}
Nodes $e_b$, $d_{2b}$ are both adjacent to nodes $c_2$ and $b$.
Node $c_2$ cannot reach outside the interval $(c_1,b)$ because
of the color-1 path $(a,1,b)$, color-2 path $(a,2,b)$ and the color-3 edge $(c_1,b)$.
Node $b$ cannot reach the interval $(c_1,a)$ because of the edges $(1,a), (2,b)$ and
$(c_1,c_2)$. Hence $e_b$ and $d_{2b}$ must be in the interval $(a,b)$.

Suppose that $e_b$ is in $(a,c_2)$, to derive a contradiction.
Then the edge $(e_b,b)$ must have color 1 and
$(c_1,e_b)$ color 3, hence the edge $(a,c_2)$ must have color 2.
Then there is no legal position for node $d_{2a}$: 
$c_2$ cannot reach outside the interval $(a,b)$ (because of the color-1 path $(a,1,b)$,
the color-2 path $(a,2,b)$ and the color-3 path $(e_b,c_1,b)$),
node 2 cannot reach the interval $(a,c_2)$ with color 2,
and node $a$ cannot reach the interval $(c_2,b)$ (because of
the color-1 edge $(e_b,b)$, the color-2 edge $(c_2,2)$ and the color-3 edge
$(c_1,e_b)$).
We conclude that $e_b$ is in interval $(c_2,b)$.

The subgraph of $Q_2$ bounded by the cycle $(c_1, c_2, 2, b)$ contains a copy of $Q_1$
with $c_1, 2$ as the outer terminals and $c_2, b$ as the inner terminals.
The inner terminals $c_2, b$ are adjacent, they lie in the same ($c_1,2$) arc
(in the cyclic order),
node $c_1$ can reach the interval $[c_2,b]$ only with color 3 and node $2$
can reach $[c_2,b]$ only with color 2. The conditions of Lemma \ref{lem:inside}
are satisfied, therefore the interval $(c_2,b)$ cannot contain
both centers $e_b, d_{2b}$.
Since $e_b$ is in the interval $(c_2,b)$, the other center $d_{2b}$ is not,
thus it is in the interval $(a,c_2)$.
The edge $(b,d_{2b})$ must have color 1 (because of the conflicting edges
$(2,c_2)$ and $(c_1,c_2)$), and the edge $(2,d_{2b})$ has color 2,
therefore the edge $(a,c_2)$ must have color 3.
\end{proof}

\begin{claim}
Node $e_a$ is in the interval $(c_1,a)$. Node $d_{1a}$ is in interval $(u,c_1)$.
Edge $(a,d_{1a})$ has color 2.
\end{claim}
\begin{proof}
Node $e_a$ is adjacent to nodes $a, c_1, c_2$.
Node $c_2$ can only reach the interval $(c_1,b)$. 
Node $a$ cannot reach the subinterval $(c_2,b)$
(because of the color-1 edge $(d_{2b},b)$, the color-2 edge $(c_2,2)$ 
and the color-3 edge $(c_1,c_2)$)
and node $c_1$ cannot reach the subinterval $(a,c_2)$
(because of the edges $(a,1), (a,2), (a,c_2)$).
Therefore, $e_a$ is in the interval $(c_1,a)$.

The subgraph of $Q_2$ bounded by the cycle $(1, a, c_2, c_1)$
contains a copy of $Q_1$ with outer terminals $1, c_2$, 
and inner terminals $a, c_1$. The inner terminals are adjacent
and are in the same $(1,c_2)$ arc. Furthermore, all edges
from $1$ to the interval $[c_1,a]$ are colored 1, and
all edges from $c_2$ to $[c_1,a]$ must be colored 3.
The conditions of Lemma \ref{lem:inside} are satisfied,
therefore we cannot have both centers $e_a, d_{1a}$ in the interval
$[c_1,a]$. Since $e_a$ is in the interval, it follows that
$d_{1a}$ must be outside the interval $[c_1,a]$.
Node $a$ cannot reach outside the interval $(u,c_2)$
(because of the color-1 path $(c_1, 1, b, d_{2b})$, the color-2 edge $(c_2,2)$
and the the color-3 edge $(c_1,c_2)$),
and node $c_1$ cannot reach the interval $(a,c_2)$.
Therefore, $d_{1a}$ must be in the interval $(u,c_1)$.
It follows that the edge $(a,d_{1a})$ must have color 2,
since it conflicts with edges $(1,c_1)$ and $(c_1,c_2)$.
\end{proof}

\begin{claim}
Node $d_{1b}$ is in interval $(b,v)$.
Edge $(c_1,d_{1b})$ is colored 3.
\end{claim}
\begin{proof}
Node $d_{1b}$ is adjacent to $b, 1, c_1$.
Node $b$ cannot reach the interval $(c_1,a)$,
node $c_1$ cannot reach the interval $(a,c_2)$ and
1 cannot reach $(c_2,b)$ (because of the edge $(d_{2b},b)$).
Therefore $d_{1b}$ lies outside the interval $(c_1,b)$.
First, we claim that it must be in the interval $(d_{1a},v)$.
Suppose to the contrary that it lies outside the interval $(d_{1a},v)$
(possibly even outside the interval $(u,v)$).
Then it is easy to see that both edges $(d_{1b},c_1)$ and $(d_{1b},b)$ must have color 3:
If $d_{1b}$ is outside the interval $(u,v)$ then this holds
because of condition (3), and if $d_{1b}$ is in $(u,d_{1a})$ this holds
because of the color-1 edge $(1,d_{1a})$ and the color-2 edges
$(d_{1a},a)$ and $(2,a)$.
Then there is no legal position to place the center $f$ of the
triangle $(b,c_1,d_{1b})$: If $f$ is outside $(u,v)$
then $(f,b), (f,c_1)$ must both have color 3 and 
one of them intersects one of $(d_{1b},c_1)$, $(d_{1b},b)$. 
Similarly, if $f$ is in $(u,c_1)$, the edge $(f,b)$ must have color 3
and intersects $(d_{1b},c_1)$. If $f$ is in the interval $(c_1,b)$ then
$(d_{1b},f)$ must be colored 3 (by condition (3) or
because of the edges $(1,d_{1a}), (d_{1a},a),  (2,a)$) 
and it intersects the edge $(c_1,b)$. 
If $f$ is in the interval $(b,v)$ then $(c_1,f)$ must have color 3
(because of the edges $(1,a), (2,a)$) and intersects the edge $(b,d_{1b})$.
We conclude that $d_{1b}$ is in the interval $(d_{1a},v)$.
Hence it is either in $(d_{1a},c_1)$ or in $(b,v)$.

Suppose that $d_{1b}$ is in the interval $(d_{1a},c_1)$,
to derive a contradiction.
The edge $(d_{1b},1)$ is colored 1 and the edge $(d_{1b},b)$
must be colored 3, hence the edge $(d_{1a},c_1)$ must be colored 2.
Consider the subgraph of $Q_2$ bounded by the cycle
$(a, c_2, c_1, d_{1a})$ and recall that all the triangles
in Figure \ref{fig:q2} are stellated twice. 
The subgraph contains a copy of $Q_1$
with $c_2, d_{1a}$ as the outer terminals,
$a, c_1$ as the inner terminals, which are adjacent,
and are embedded inside the interval $(d_{1a},c_2)$.
Node $d_{1a}$ can reach the interval $[a,c_1]$ only with color 2
(because of the color-1 edge $(d_{1b},1)$ and the color-3
edge $(d_{1b},b)$), and $c_2$ can reach $[a,c_1]$ only with color 3.
The center $e_a$ of the triangle $(c_2, c_1, a)$ is inside the
interval $(a,c_1)$.
Hence, by Lemma \ref{lem:inside}, the center $g$ of the
other triangle $(d_{1a}, c_1, a)$ cannot be in the
interval $(c_1,a)$.
This leaves no possible position for the center $g$ of $(d_{1a}, c_1, a)$:
Node $c_1$ cannot reach outside $(u,v)$ or left of $d_{1a}$
(because of the edges $(d_{1b},1), (d_{1a},a), (d_{1b},b)$ and condition 3),
node $a$ cannot reach inside the interval $(d_{1a},c_1)$ or $(b,v)$,   
and $d_{1a}$ cannot reach the interval $(a,b)$.

We conclude that $d_{1b}$ is not in the interval $(d_{1a},c_1)$,
hence it must be in the interval $(b,v)$.
The edge $(c_1,d_{1b})$ must have color 3 because
it conflicts with the edges $(a,1), (a,2)$.
\end{proof}

This completes the proof of the lemma.
\end{proof}

Let $Q$ be the graph formed by taking 15 copies of $Q_2$,
identifying their outer terminals 1,2, and identifying terminal $b$
of the $i$-th copy with terminal $a$ of the $(i+1)$-th copy;
i.e., $Q$ is formed by glueing together back-to-back 15 copies of $Q_2$
with the same outer terminals, see Figure \ref{fig:q}. 
We call $Q$ a {\em quad}, nodes 1,2 the outer terminals of quad $Q$
and call the inner terminals $a_1, a_2, \ldots, a_{16}$ of the copies of $Q_2$
the inner terminals of $Q$.
We use $Q_2^i$ to denote the $i$th copy of $Q_2$.
Let $Q+12$ denote the graph consisting of $Q$ and the edge $(1,2)$.

\begin{figure}[h]
\centering
\vspace*{-1.2cm}
\includegraphics[scale=0.6]{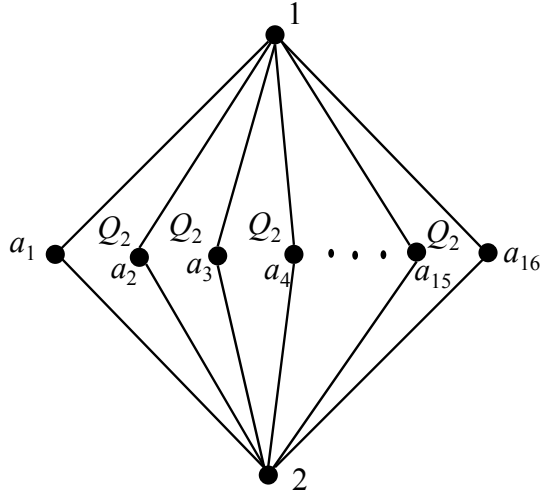}
\vspace*{-3.8cm}
\caption{The quad $Q$.}
\vspace*{-0.2cm}
\label{fig:q}
\end{figure}

\begin{lemma}\label{lem:quad}
There is no embedding of $Q+12$ in three pages such that
all the inner terminals are embedded in a subinterval of the interval $(1,2)$,
and all edges from nodes 1 and 2 to this subinterval use only two (the same two) colors (i.e., one of the three colors is not used by any
edge connecting nodes 1 and 2 to this subinterval).
\end{lemma} 
\begin{proof}
Consider a book embedding of the nodes of the quad $Q$
with all the inner terminals embedded
in a subinterval of the interval $(1,2)$ and edges from 1, 2 
to this subinterval using two colors. 
Let $u', u$ be the inner terminals closest to 1
and $v', v$ the inner terminals closest to 2; thus, if we assume wlog that
the embedding on the line starts with terminal 1, 
then $u',u$ are the first two inner terminals and $v, v'$ the last two.
The order of these nodes is $1, u', u, v, v', 2$.
Let 1 be the color of the edge $(1,v')$ and let 2 be the color of the edge $(2,u')$.
By the hypothesis, all edges from nodes 1 and 2 to the interval $(u',v')$
are colored 1 or 2. Since the edge $(2,u')$ is colored 2,
all edges from node 1 to the interval $[u,v]$
(including nodes $u, v$)
must be colored 1.
Similarly, since the edge $(1,v')$ is colored 1,
all edges from node 2 to the interval $[u,v]$ (including $u$ and $v$) must be colored 2.
Because of the color-1 path $(u,1,v)$ and the color-2 path $(u,2,v)$,
all other edges that exit the interval $(u,v)$ (and are not going to 1,2)
must be colored 3.
Thus, the conditions of Lemma \ref{lem:in-out} are satisfied for every
copy of $Q_2$ whose inner terminals are in the interval $(u,v)$.
Since there are 16 inner terminals, there are at least three consecutive
terminals $a_i, a_{i+1}, a_{i+2}$ that are not in the set $\{u,u',v,v' \}$,
i.e. that are embedded in the interval $(u,v)$.
Thus, there are two consecutive copies of $Q_2$, for which
Lemma \ref{lem:in-out} holds.

Consider any copy of $Q_2$ with inner terminals in the interval $(u,v)$.
Observe from the conclusion of Lemma \ref{lem:in-out} (see Fig. \ref{fig:in-out})
that the inner terminals $a,b$ are connected by paths of all 3 colors,
$(a,1,b), (a,2,b), (a, c_2, c_1, b)$, and thus no edge can exit the interval
$(a,b)$ unless it connects to $1, 2$, $c_1$ or $c_2$.
Second, observe that there are nodes of $Q_2$ outside the interval $(a,b)$
on both sides, e.g. the nodes $c_1, d_{1b}$.
Applying these observations to the $i$-th and $(i+1)$-th copy of $Q_2$,
tells us that $a_i$ cannot lie inside the interval $(a_{i+1}, a_{i+2})$,
and similarly $a_{i+2}$ cannot lie inside the interval $(a_i, a_{i+1})$.
For, suppose to the contrary that $a_i$ is in $(a_{i+1}, a_{i+2})$.
Since no edge of the $i$-th copy $Q_2^i$ can exit the interval $(a_{i+1}, a_{i+2})$,
all nodes of $Q_2^i$ must lie inside the interval $(a_{i+1}, a_{i+2})$,
contradicting the fact that some nodes must lie on both sides outside $(a_i,a_{i+1})$. 
Therefore, $a_i$ is outside the interval $(a_{i+1}, a_{i+2})$.
Similarly $a_{i+2}$ is outside the interval $(a_i, a_{i+1})$.

Thus, we may assume without loss of generality that the nodes $a_i, a_{i+1}, a_{i+2}$
appear in this order.
Observe from Lemma \ref{lem:in-out} for each of copy of $Q_2$ that
there is a color-3 edge $e$ (edge $(c_1,d_{1b})$ in Fig. \ref{fig:in-out}) that connects
two nodes of $Q_2$ that lie outside, and on both sides, of the interval $(a,b)$
and that furthermore, one of these two nodes of the edge $e$ 
(node $c_1$ in Fig. \ref{fig:in-out})
has a color-3 edge to an inner terminal ($b$ in Fig. \ref{fig:in-out}).
Let $e^i$ be the edge $e$ for $Q_2^i$, and
$ e^{i+1}$ for $Q_2^{i+1}$.
The left endpoint of $e^i$ is left of $a_i$ 
and the right endpoint is right of $a_{i+1}$, and since it cannot be
in the interval $(a_{i+1}, a_{i+2})$ (no edge that does not belong to $Q_2^{i+1}$ can exit this interval unless it goes to 1 or 2), it must be right of $a_{i+2}$.
Similarly, the left endpoint of $e^{i+1}$ must be left of $a_i$
and the right endpoint right of $a_{i+2}$.
Since the edges $e^i, e^{i+1}$ both have color 3, they are nested.
Suppose without loss of generality that $e^i$ is nested inside $e^{i+1}$.
Then neither endpoint of $e^{i+1}$ can have a color-3 edge to 
$a_{i+1}$ or $a_{i+2}$, because it would conflict with $e^i$.
This contradicts Lemma  \ref{lem:in-out}.
\end{proof}

We remark that the properties of Lemmas \ref{lem:in-out} and \ref{lem:quad}
depend crucially on the fact that in $Q_2$ we included the edge $(c_1,c_2)$ between the
two centers rather than the edge $(a,b)$ between the two inner terminals.
It can be shown that with the edge $(a,b)$ instead, even after stellating
the faces an arbitrary number of times, and gluing back-to-back an arbitrary number
of copies of the resulting graph, yields a graph that does not 
have the property of Lemma \ref{lem:quad}:
the graph can be embedded between the two outer terminals
so that all edges to each outer terminal have the same color.

Attaching quads to the edges of a graph restricts the possible embeddings into 3 pages. 
The following lemma illustrates how $Q$ can be used.
The lemma will be used often in the sequel.

\begin{lemma}\label{lem:triangle}
Consider a graph $H$ formed by taking a triangle $(A,B,C)$ on the plane,
attaching (adding) a quad to each edge of the triangle, 
by identifying the outer terminals of the quad with the nodes of the edge.
There is no embedding of $H$ in three pages such that
all inner terminals of the quads are embedded in the arc $(B,C)$ that does not
contain $A$ and all edges from $A$ to the inner terminals have the same color.
\end{lemma}
\begin{proof}
Suppose that there is such a 3-page embedding, consider its linearization
with $A$ lying outside the interval $(B,C)$.
Assume first that among all the inner terminals of the quads, the one embedded
closest to $B$ does not belong to the $AB$ quad, i.e., it belongs to
the $AC$ or the $BC$ quad. Let $u$ be this terminal and assume wlog that
the edge $(u,C)$ has color 1. All the inner terminals of the $AB$ quad 
are in the interval $(u,C)$; this interval can be reached from $A$ and $B$
only with colors 2 and 3. By Lemma \ref{lem:quad} this is impossible.

Therefore, the inner terminal $u$ closest to $B$ belongs to the $AB$ quad.
By a symmetric argument, the inner terminal $v$ closest to $C$ belongs to the $AC$ quad.
The edges $(A,u)$ and $(A,v)$ have the same color, say color 1.
Then all inner terminals of the $BC$ quad are in the interval $(u,v)$
and $B$ and $C$ can reach this interval only with colors 2 and 3.
This is impossible by Lemma \ref{lem:quad}.
\end{proof}

\section{The Graph}

Our `hard' graph $G$ is constructed as follows.
Take a long path $p=(x_1, x_2, \ldots, x_n)$, of $n$ nodes,
where $n$ is sufficiently large, say $n=1000$.
Take two other nodes 1, 2 and connect them to
all the nodes $x_i$ of the path, as in Fig. \ref{fig:graph}. 
This forms  2$(n-1)$ triangles, which we call the {\em big triangles}.
Subdivide each big triangle to three {\em small triangles}
by inserting a {\em center} node and connecting it to the 3 nodes of the triangle.
Inside each small triangle, attach a copy of the quad $Q$ to
each edge of the triangle, 
and add a central node connecting it to the
three nodes of the triangle and the innermost inner terminals of the
three quads. 
The construction is shown in Figure \ref{fig:graph}
(except for the quads attached to the edges, which we omitted
for clarity).
Note that all small triangles as well as all big triangles
satisfy the conditions of Lemma \ref{lem:triangle}.
Note also the copies of $Q_1$ with outer terminals 1,2 and inner terminals $x_i, x_{i+1}$.
We call the nodes 1, 2 the {\em terminals} of $G$,
we call the $n$ nodes $x_1, \ldots, x_n$ the
{\em vertical nodes} of $G$, and call the edges $(x_i,x_{i+1})$
the {\em vertical edges}.

\begin{figure}[h]
\centering
\vspace*{-1cm}
\includegraphics[scale=0.6]{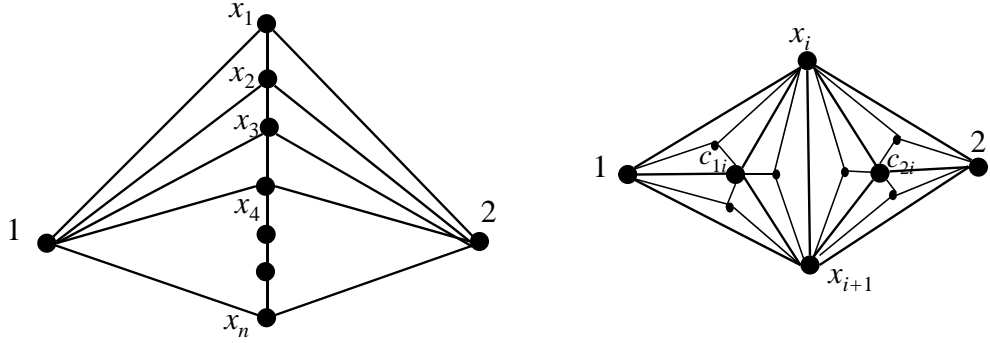}
\vspace*{-6.4cm}
\caption{The Graph $G$. (There are quads attached to the edges.)}
\vspace*{-0.3cm}
\label{fig:graph}
\end{figure}

We will show that $G$ cannot be embedded in three pages.
For this purpose, fix any 3-page embedding of $G$.
We will show that the embedding has to satisfy a sequence
of properties, and derive eventually a contradiction.

Given a 3-page circle embedding of $G$,
we designate one of the two arcs 
between the two terminals 1,2 as the {\em major} (1,2) arc,
and the other as the {\em minor} (1,2) arc as follows:
the major arc is an arc (1,2)
that contains at least half of the vertical nodes,
and the other (1,2) arc is the {\em minor} arc
(if both (1,2) arcs contain exactly half of the vertical nodes,
we arbitrarily designate one as the major and the other
as the minor arc).
Let $z_1$ be the node in the major arc that is closest to node 1 
and adjacent to 2,
and let $z_2$ be the node in the major arc that is closest to node 2
and adjacent to 1
 - see Fig. \ref{fig:crossing}.
We show first that there are not many vertical edges $(x_i,x_{i+1})$
with nodes on both (1,2) arcs.

\begin{figure}[h]
\centering
\vspace*{-0.8cm}
\includegraphics[scale=0.6]{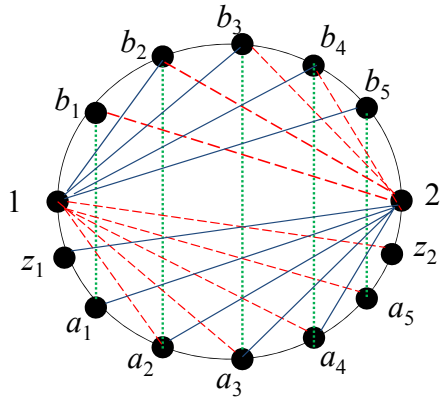}
\vspace*{-6cm}
\caption{Crossing vertical edges.}
\vspace*{-0.25cm}
\label{fig:crossing}
\end{figure}

\begin{lemma}\label{lem:crossing}
There are at most 4 vertical nodes $x_i$ in the arc $(z_1,z_2)$, whose
successor $x_{i+1}$ on the path $p$ is in the minor arc (1,2).
\end{lemma}
\begin{proof}
By contradiction.
Suppose that there are 5 such vertical nodes $x_i$ in $(z_1,z_2)$
whose successor $x_{i+1}$ is in the minor arc (1,2).
Denote these 5 nodes as $a_1, \ldots , a_5$ in the order that they appear in $(z_1,z_2)$, and let their successors be $b_1, \ldots ,b_5$ respectively.
Assume wlog that edge $(1,z_2)$ has
color 1 and edge $(2,z_1)$ has color 2 
(they intersect so they must have different colors).
The 5 vertical edges $(a_1,b_1), \ldots, (a_5,b_5)$ intersect both edges $(1,z_2), (2,z_1)$, so they must all have color 3,
hence they do not intersect each other; see Fig. \ref{fig:crossing}.
The edges $(1,a_2), \ldots (1,a_5)$ must have color 1 because they intersect $(a_1,b_1)$ and $(2,z_1)$. Similarly the edges $(2,a_1), \ldots (2,a_4)$ must have color 2 because they intersect $(a_5,b_5)$ and $(1,z_2)$.
Note that a node inside the arc $(a_2 , a_3)$ can reach nodes other than nodes 1,2 only in the
same arc or the arc $[b_2,b_3]$,
 because of the color-1 path $(a_2,1,a_3)$, the color-2 path
 $(a_2,2,a_3)$ and the color-3 edges $(a_2,b_2), (a_3,b_3)$.
Similarly, a node inside the arc $(a_3,a_4)$ can reach nodes other than nodes 1,2 only in the
same arc or the arc $[b_3,b_4]$.
Node $a_3$ can only reach nodes in the arcs $[a_2,a_3,a_4]$ and 
$[b_2,b_3,b_4]$.

Similar observations hold for the
edges connecting nodes 1, 2 to the
vertical nodes $b_1, \ldots , b_5$ on the minor arc (1,2).
Edges $(1,b_5), (2,b_1)$,  must have color 1 or 2 because they intersect $(a_2,b_2)$.
If $(1,b_5)$ has color 1 then $(2,b_1)$ must have color 2,
hence all edges $(1,b_2), \ldots (1,b_5)$ must have color 1
and all the edges $(2,b_1), \ldots (2,b_4)$ must have color 2.
If $(1,b_5)$ has color 2 then $(2,b_1)$ must have color 1,
all edges $(1,b_2), \ldots (1,b_5)$ have color 2
and the edges $(2,b_1), \ldots (2,b_4)$ have color 1.
(Figure \ref{fig:crossing} depicts the latter case.)
In either case, note again that a node in the arc $(b_2,b_3)$
can reach nodes other than 1,2 only in the
same arc or the arc $[a_2,a_3]$.
Similarly, a node inside the arc $(b_3,b_4)$ can reach nodes other than 1, 2 only in the
same arc or the arc $[a_3,a_4]$.
Node $b_3$ can only reach nodes in the arcs $[b_2,b_3,b_4]$ and $[a_2,a_3,a_4]$.

Consider the big triangle $(1,a_3,b_3)$ of $G$.
From the observations in the previous two paragraphs it follows
that either all the internal nodes of this triangle are in the strip $(a_2,a_3) \cup (b_2,b_3)$ or they are all in the strip $(a_3,a_4) \cup (b_3,b_4)$. Assume without loss of generality that they are in the strip $(a_2,a_3) \cup (b_2,b_3)$ (the argument is the same in the other case).
Any edge from $a_3$ to the arc $(b_2,b_3)$ must be colored 3
and likewise any edge from $b_3$ to the arc $(a_2,a_3)$ must be colored 3, therefore there cannot exist both kinds of edges.
Hence, either all the inner terminals of the quads attached to
edges $(1,a_3), (a_3,b_3)$ are in arc $(a_2,a_3)$
or all inner terminals of the quads attached to the edges
$(1,b_3), (a_3,b_3)$ are in arc $(b_2,b_3)$.
Assume wlog that the former holds, i.e. all the inner terminals of the quads attached to
edges $(1,a_3), (a_3,b_3)$ are in arc $(a_2,a_3)$.

If the inner terminal closest to $a_3$ belongs to the $(1,a_3)$
quad, then the edge connecting it to 1 must be colored 1, hence 
all edges from the inner terminals of the $(a_3,b_3)$ quad to $a_3$
must be colored 2 or 3 and all edges from these inner terminals to
$b_3$ must be colored 3, contradicting Lemma \ref{lem:quad}.
On the other hand, if the inner terminal closest to $a_3$
belongs to the $(a_3,b_3)$ quad, then the edge to $b_3$ is colored 3,
hence 
all edges from the inner terminals of the $(1,a_3)$ quad to $a_3$
must be colored 1 or 2 and all edges from these inner terminals to
$1$ must be colored 1, contradicting again Lemma \ref{lem:quad}.
The lemma follows.
\end{proof}

We view the 3-page embedding as an embedding on the line, starting with one 
of the terminals, followed by the major (1,2) arc (which is now an interval on the line), then the other terminal, followed by the minor arc (1,2).
We will focus on the interval $(1,2)$ corresponding to the major arc.
Let $z_1$ (resp. $z_2$) be again the node in the interval $(1,2)$
that is closest to node 1 (resp. 2)
and is adjacent to 2 (resp. 1).
Define the {\em stretch} between two points $u, v$, of the linear embedding, 
denoted $str(u,v)$, to
be the number of vertical nodes in the interval $(u,v)$.

\begin{lemma}\label{lem:stretch}
(1) If $u, v$ are two nodes in the interval $(1,2)$ connected by an edge 
then $str(u,v) \leq 15$.\\
(2) If $u, v$ are two nodes in the interval $(z_1,z_2)$ 
that have incident edges that exit the interval $[1,2]$
then $str(u,v) \leq 15$.
\end{lemma}
\begin{proof}
The proof is essentially the same for both parts.
Suppose that $u,v$ are two nodes as in part (1) or (2) of the lemma
such that $str(u,v) > 15$. We will derive a contradiction.
If there is an edge $(u,v)$, say of color 3,  then all edges from the terminals 1, 2
to all the nodes in the interval $(u,v)$ must use the other two
colors 1, 2.
On the other hand, if $u , v$ are in the interval $(z_1,z_2)$ and there are two edges incident to them that exit the
interval $[1,2]$, then both edges intersect the edges
$(1,z_2)$ and $(2,z_1)$, which themselves intersect each other.
Therefore, the two exiting edges incident to $u$ and $v$ must have the same color,
say color 3, and thus again all edges from the terminals 1, 2
to all the nodes in the interval $(u,v)$ must use the other two
colors 1, 2.

Let $y_1, y'_1$ (respectively, $y_2, y'_2$)
be the vertical nodes in the interval $(u,v)$ 
that are closest to $1$ (resp.  to $2$).
So, if we assume wlog that $u$ is closer to 1 (than $v$ is)
and $v$ is closer to 2,
then the order of these nodes is
$1, u, y_1, y'_1, y'_2, y_2, v, 2$ (or the reverse).
The other vertical nodes in the interval $(u,v)$ lie
between $y'_1$ and $y'_2$.
Assume without loss of generality that $(1,y_2)$ has color 1 and
$(2,y_1)$ has color 2. Then all edges from 1 to the vertical nodes
in the interval $(u,v)$, except possibly $y_1$ have color 1,
and  all edges from 2 to the vertical nodes in $(u,v)$, except possibly $y_2$, have color 2.

Let $x_i$ be any vertical node in the interval $(u,v)$ such that
$x_i,\notin \{ y_1, y'_1, y_2, y'_2, x_1, x_{n}\}$, 
$x_{i+1} \notin \{ y_1, y'_1, y_2, y'_2, x_{n}\}$ and $x_{i+1}$
is in the interval $(1,2)$ (i.e., $x_{i+1}$ is not in the minor arc (1,2)). 
By Lemma \ref{lem:crossing} there are
at most 4 vertical nodes $x_i$ in the interval $(z_1,z_2)$ whose successor
$x_{i+1}$ is in the minor arc (1,2). 
Since there are at least 16 vertical nodes in $(u,v)$,
we can choose $x_i$ to be a vertical node in $(u,v)$ that satisfies the
above conditions.

\begin{claim}\label{cl:next1}
Node $x_{i+1}$ is in $(u,v)$ and there is no other vertical node
embedded between $x_i$ and $x_{i+1}$.
\end{claim}
\begin{proof}
Suppose that $x_{i+1}$ is not in $(u,v)$.
Then the edge $(x_i,x_{i+1})$ must have color 3 because it conflicts with the color-1
path 
$(y'_1,1,y'_2)$ and the color-2 path $(y'_1,2,y'_2)$. 
In case (1) of the statement of the lemma, where there is an edge $(u,v)$ of color 3,
we get a contradiction. In case (2) where there are color-3 edges incident to $u$
and $v$ that exit the interval $(1,2)$, there is again a contradiction because
$x_{i+1}$ lies in the interval $(1,2)$,
and therefore $(x_i,x_{i+1})$ intersects one of the exiting edges 
incident to $u, v$.
In either case we conclude that $x_{i+1}$ is in $(u,v)$.

\begin{figure}[h]
\centering
\vspace*{-1.2cm}
\includegraphics[scale=0.6]{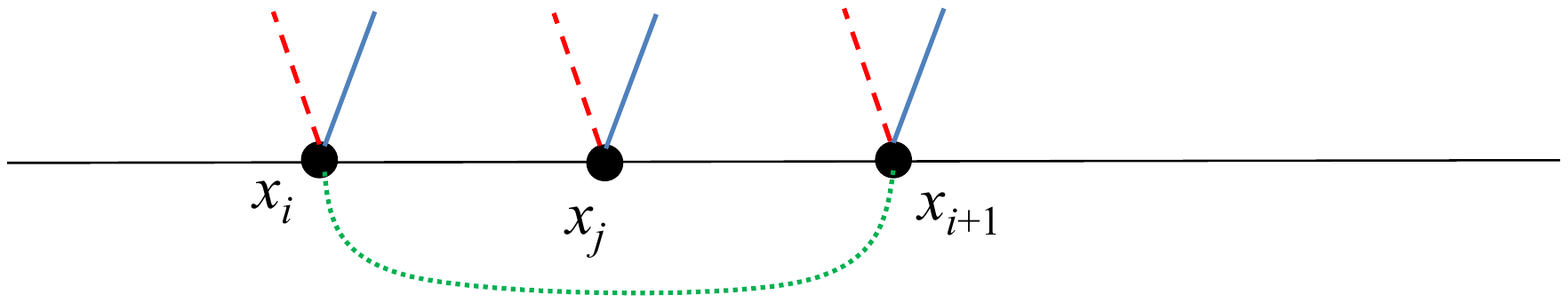}
\vspace*{-8.7cm}
\caption{}
\vspace*{-0.3cm}
\label{fig:next1}
\end{figure}

Since neither $x_i$ nor $x_{i+1}$ is $y_1$ or $y_2$,
the edges $(1,x_i), (1,x_{i+1})$ are colored 1 and the
edges $(2,x_i), (2,x_{i+1})$ are colored 2.
Suppose that there is another vertical node $x_j$,
in the interval between $x_i$ and $x_{i+1}$ (see Fig. \ref{fig:next1}).
Then the edge $(1,x_j)$ has color 1, and the edge $(2,x_j)$ has color 2.
Therefore, the edge $(x_i,x_{i+1})$ must have color 3.
No edge can exit the interval $(x_i,x_{i+1})$, other than edges to 1, 2,
because of the color-1 path $(x_i, 1, x_{i+1})$, the color-2 path
$(x_i, 2, x_{i+1})$ and the color-3 edge $(x_i,x_{i+1})$.
Since $x_j$ is in this interval, then all vertical nodes of the path $p$
between $x_j$ and $x_i$ or $x_{i+1}$ must be in this interval.
That is, if $j<i$, then $x_{j+1}, \ldots, x_{i-1}$ must be in the interval
$(x_i,x_{i+1})$; 
if $j>i+1$, then $x_{i+2}, \ldots, x_j$ are in the interval $(x_i,x_{i+1})$.

Suppose without loss of generality that $j<i$, hence $x_{i-1}$ is in
$(x_i,x_{i+1})$. 
Consider the triangles $(1,x_{i-1},x_{i})$, $(2,x_{i-1},x_{i})$.
Their interior nodes must all be in the interval $(x_i,x_{i+1})$
because $x_{i-1}$ is in this interval
(refer to Fig. \ref{fig:next1} with $x_{i-1}$ in
place of $x_j$).
If all the inner terminals of the quad attached to the edge $(1,x_i)$
are in the interval $(x_{i-1}, x_{i+1})$, then all edges connecting them
to $1$ must have color 1 and all edges to $x_i$ must have color 3,
contradicting Lemma \ref{lem:quad}.
Similarly, if all the inner terminals of the quad attached to the edge $(2,x_i)$
are in the interval $(x_{i-1}, x_{i+1})$, then all edges connecting them
to $2$ must have color 2 and all edges to $x_i$ must have color 3,
contradicting again Lemma \ref{lem:quad}.
It follows that at least one inner terminal from each of the two
quads attached to edges $(1,x_i), (2,x_i)$ must be in
the interval $(x_i, x_{i-1})$; the edges connecting these
inner terminals to 1 and 2 respectively are colored 1, 2.
Hence edge $(x_i, x_{i-1})$ is colored 3.
Because of the color-1 and -2 paths $(x_i,1,x_{i-1})$ and
$(x_i,2,x_{i-1})$, no edge to a node other than 1, 2, can exit 
the interval $(x_i, x_{i-1})$, hence all the interior of the
triangle $(1,x_{i-1},x_{i})$ (as well as $(2,x_{i-1},x_{i})$) must lie
in this interval. All the edges from node 1 to the
inner terminals must use color 1, so this contradicts Lemma \ref{lem:triangle}. 
\end{proof}

Assume without loss of generality that $x_{i+1}$ is embedded right of $x_i$.
Let $y$ be the vertical node left of $x_i$ and $z$ the vertical node
right of $x_{i+1}$.
Since $x_i, x_{i+1}$ are in $(u,v)$ and are not
among $\{ y_1, y'_1, y_2, y'_2\}$, the nodes $y, z$ exist,
and the edges from 
$y, x_i, x_{i+1}, z$ to 1 are colored 1 and the edges to 2 are colored 2.

\begin{claim}\label{cl:next2}
If a node adjacent to both $x_i, x_{i+1}$ (other than 1, 2) is outside the
interval  $(x_i, x_{i+1})$ then it must be either in the interval
$(y,x_i)$ or in the interval $(x_{i+1},z)$.
In the former case $y=x_{i-1}$ and in the latter case $z=x_{i+2}$.
\end{claim}
\begin{proof}
Let $w$ be a node adjacent to both $x_i, x_{i+1}$ (other than 1, 2) 
that is outside the interval  $(x_i, x_{i+1})$.
Note that $w \neq z$, since $w$ is adjacent to both $x_i, x_{i+1}$,
hence it is not a vertical node, and $z$ is a vertical node.
Assume without loss of generality that $w$ is to the right of the interval  $(x_i, x_{i+1})$.
(The argument is symmetric if $w$ is to the left of the interval.)
If $w$ is right of $z$, then both edges $(z,x_i), (z,x_{i+1})$ have color 3
(because they intersect the color-1 and -2 paths $(y,1,z), (y,2,z)$),
thus there is a color-3 path between $x_i$ and $x_{i+1}$.
There are also color-1 and -2 paths $(x_i,1,x_{i+1})$, $(x_i,2,x_{i+1})$,
hence no edge can exit the interval $(x_i, x_{i+1})$ to a node other than 1, 2, $w$,
and any other node adjacent to both $x_i, x_{i+1}$ must be inside $(x_i, x_{i+1})$.
This implies in particular that the center of either the triangle
$(1,x_i,x_{i+1})$ or the triangle  $(2,x_i,x_{i+1})$ is in the
interval $(x_i, x_{i+1})$, and hence all the nodes of the triangle
are in this interval.
This contradicts Lemma \ref{lem:triangle} because the interval
is reachable from node 1 only with color 1 (and from node 2 only with color 2).

We conclude that $w$ is in the interval $(x_{i+1},z)$ - see Fig. \ref{fig:next2}.
The edge $(x_i,w)$ has color 3 because it intersects $(1,x_{i+1}), (2,x_{i+1})$.
Node $x_{i+1}$ cannot reach any node other than 1, 2 left of $x_i$ nor right of $z$ because of the color-1 and -2 paths $(x_i,1,z), (x_i,2,z)$ and the color-3 edge
 $(x_i,w)$. 
Since $x_{i+1} \neq x_n$ from our choice of $x_i$,
node $x_{i+1}$ has a successor $x_{i+2}$ on the path $p$,
and $x_{i+2}$ must lie in the interval $(x_i,z)$.
By Claim \ref{cl:next1}, $x_{i+2}$ cannot be in the interval  $(x_i,x_{i+1})$,
hence it must be $z$.
\end{proof}

\begin{figure}[h]
\centering
\vspace*{-1.2cm}
\includegraphics[scale=0.6]{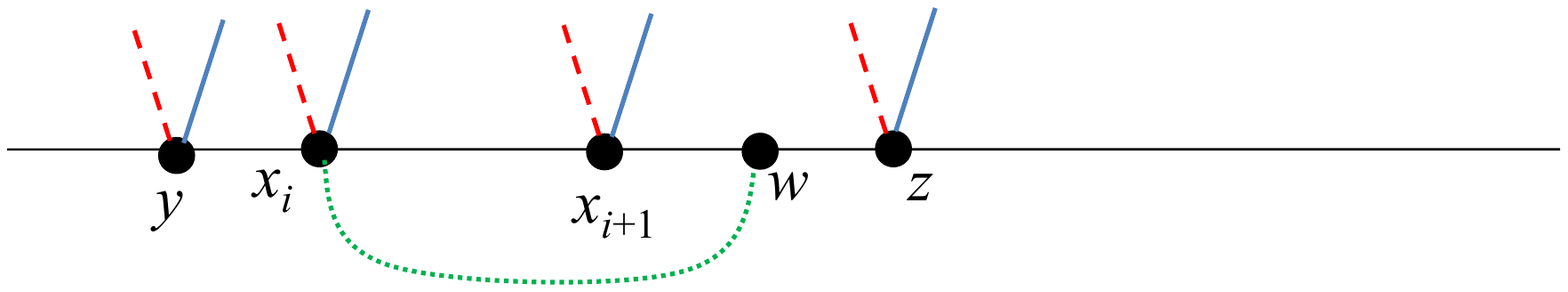}
\vspace*{-9.2cm}
\caption{}
\vspace*{-0.3cm}
\label{fig:next2}
\end{figure}

We can finish now the proof of the lemma.
Applying Lemma \ref{lem:inside} to the copy of $Q_1$ with outer terminals
$1, 2$ and inner terminals $x_i, x_{i+1}$,
we deduce that at least one of the centers
of the triangles $(1,x_i,x_{i+1})$, $(2,x_i,x_{i+1})$ must be outside
the interval $(x_i, x_{i+1})$.
Let $w$ be this center. By Claim \ref{cl:next2}, $w$
is either in $(y,x_i)$ or in $(x_{i+1},z)$.
Assume wlog that $w$ is in in $(x_{i+1},z)$
- see Fig. 11.
Then $z=x_{i+2}$ by Claim \ref{cl:next2}.
The edge $(x_i,w)$ has color 3 (because it intersects $(1,x_{i+1}), (2,x_{i+1})$),
 hence $x_{i+1}$ cannot reach outside the interval $(x_i,x_{i+2})$
and $x_{i+2}$ cannot reach inside $(x_i,x_{i+1})$.
 Therefore, all nodes (other than 1, 2) adjacent to both $x_{i+1},x_{i+2}$,
 e.g. the centers of the triangles $(1,x_{i+1},x_{i+2})$ and $(2,x_{i+1},x_{i+2})$
 must be inside the interval $(x_{i+1},x_{i+2})$.
 This contradicts Lemma \ref{lem:inside} for the copy of $Q_1$ with outer terminals
$1, 2$ and inner terminals $x_{i+1}, x_{i+2}$.
\end{proof}

We will concentrate on a region of the interval between the terminals
that has the properties indicated in the following lemma.

\begin{lemma}\label{lem:prime}
There is a subinterval $I$ of the interval $(1,2)$ such that
(1) no edge from a node in $I$ exits the interval $(1,2)$,
(2) all edges from one terminal to $I$ use only one color, and all
edges from the other terminal use the other two colors, and
(3) $I$ contains at least 240 vertical nodes.
\end{lemma}
\begin{proof}
Let again $z_1$ (resp. $z_2$) be the node in the interval $(1,2)$ that
is closest to 1 (resp. 2) and is adjacent to 2 (resp. 1).
Assume wlog that the edge $(1,z_2)$ 
has color 1 and the edge $(2,z_1)$ has color 2.
All edges from 1 to nodes in the interval $(z_1,2)$
have color 1 or 3,
and similarly all edges from 2 to nodes in $(1,z_2)$
have color 2 or 3. 
Let $g_2$ be the node in the interval $(z_1,z_2)$ closest to 2 that has a color-3 edge to 1
if there is such a node (see Fig. \ref{fig:prime}); otherwise let $g_2=z_1$.
Similarly, let $g_1$ be the node in the interval $(z_1,z_2)$ closest to 1 
that has a color-3 edge to 2, if there is such a node, otherwise let $g_1=z_2$.
From the definitions, all edges from $1$ to the (open) interval $(g_2,2)$ have color 1,
and all edges from 2 to the interval $(1,g_1)$ have color 2.
Clearly, $g_2$ must be left of (or equal to) $g_1$ because two color-3 edges cannot intersect.
If there is an edge incident to a node $u $ in the interval $(z_1,z_2)$ that
exits the interval $(1,2)$, it must have color 3 and thus $u$ must be between $g_2$ and $g_1$.
If there are such exiting edges,
then let $f_1$ be the leftmost node in $(z_1,z_2)$ that has such an edge and
$f_2$ be the rightmost such node. 
By Lemma \ref{lem:stretch}, $str(f_1,f_2) \leq 15$.
If there are no such edges that connect a node in $(z_1,z_2)$ to a node outside $(1,2)$,
then let $f_1=g_1$, $f_2=g_2$.

\begin{figure}[h]
\centering
\vspace*{-1.2cm}
\includegraphics[scale=0.6]{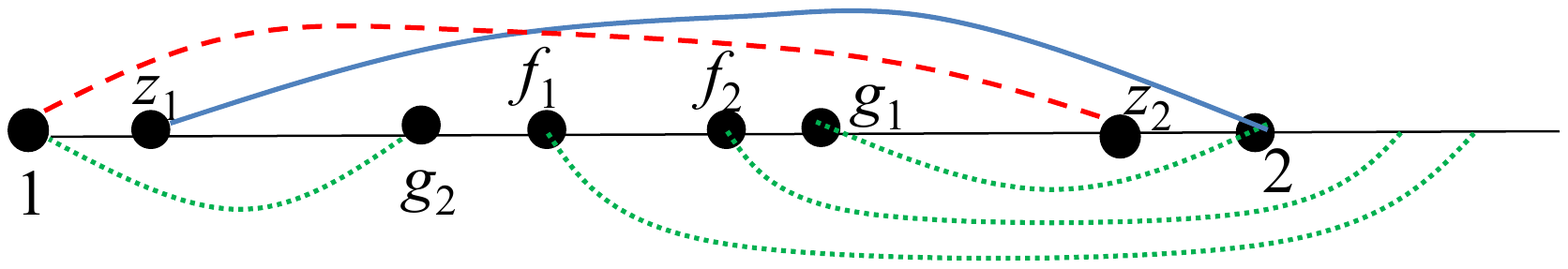}
\vspace*{-8.8cm}
\caption{}
\vspace*{-0.3cm}
\label{fig:prime}
\end{figure}

The intervals $(z_1,f_1)$, $(f_2,z_2)$ have the following properties:\\
(1) No edge from a node in these intervals exits $(1,2)$.\\
(2) All edges from node 1 to $(f_2,z_2)$ have color 1, and all edges from node 2
to $(z_1,f_1)$ have color 2.\\
(3) The total number of vertical nodes in the two intervals is at least $\frac{n}{2}-20 = 480$.

The interval among  $(z_1,f_1)$, $(f_2,z_2)$ that has the maximum number of
vertical nodes satisfies the conditions of the lemma.
\end{proof}

We call the subinterval $I$ of Lemma \ref{lem:prime} the {\em prime region}.
Assume wlog for the remainder that all edges from the prime region $I$ 
to terminal 1 have color 1,
and edges to terminal 2 have color 2 or 3; there are no edges from $I$ that exit the interval $(1,2)$.

By Lemma \ref{lem:stretch}, if a node in the prime region $I$
has stretch distance more than 15 from the endpoints of $I$,
then all its adjacent nodes (except 1,2) are also in the prime region.
Let $x_i, i \neq n$ be any vertical node in the prime region
that has stretch distance more than 50 from its endpoints.
Then all the nodes that are within distance 3 from $x_i$ in the
graph $G \setminus \{1,2\}$ are also in the prime region.
This implies in particular that
$x_{i+1}$ is also in the prime region, and so are
all nodes adjacent to $x_i$ or $x_{i+1}$ in $G$ (except 1,2)
and their adjacent nodes.
The edges from 1 to $x_i$, $x_{i+1}$ are both colored 1.
The edges from 2 to $x_i$, $x_{i+1}$ are colored 2 or 3;
we show next that they have the same color.

\begin{lemma}\label{lem:same}
Let $x_i, i \neq n$ be any vertical node in the prime region
that has stretch distance more than 50 from its endpoints.
The edges $(2,x_i), (2,x_{i+1})$ have the same color.
\end{lemma}
\begin{proof}
Suppose that the edges $(2,x_i), (2,x_{i+1})$ have different colors,
say 2, 3 respectively. We will derive a contradiction.
Assume wlog that $x_i$ is left of $x_{i+1}$.
We distinguish cases depending on the color of the edge $(x_i,x_{i+1})$.

{\em Case 1: $(x_i,x_{i+1})$ has color 1}.
Consider the inner terminals of the quads attached to the edges
$(1,x_i)$ and $(1,x_{i+1})$. They must be in the prime region
(since they are adjacent to $x_i$ or $x_{i+1}$), but they cannot be in
the interval $(x_i,x_{i+1})$ (since it has color 1).

Observe from the definition of $Q_2$ that the two inner terminals $a, b$ 
are connected by two paths of length 2, 
where the intermediate node of each path is adjacent to
one of the terminals.
Therefore, any two inner terminals of a quad $Q$ are connected by
a path consisting of nodes that are adjacent to either outer terminal,
where every other node of the path is an inner terminal.
In particular, any two inner terminals of the quads attached to the edges
$(1,x_i)$ and $(1,x_{i+1})$ are connected by a path of nodes
that are adjacent to node 1.
All of the nodes of these paths must be in the prime region
because they have distance at most 2 in the graph $G \setminus \{1,2\}$ from $x_i$ or $x_{i+1}$.
None of these nodes can be in the interval $(x_i,x_{i+1})$
because the edge $(x_i,x_{i+1})$ is colored 1.
Furthermore, there cannot be an edge connecting a node in the prime region
left of $x_i$ to a node right of $x_{i+1}$ because of the edges
$(1,x_i), (2,x_i), (2,x_{i+1})$ that have colors 1, 2, 3 respectively.
Therefore, either all the inner terminals of both quads
are left of $x_i$ or they are all right of $x_{i+1}$.
In the former case, the quad of $(1,x_{i+1})$ contradicts Lemma \ref{lem:quad}
because all edges from the inner terminals to 1 have color 1 and the edges to $x_{i+1}$
must have color 3 (because they intersect $(1,x_i), (2,x_i)$).
In the latter case, the quad of $(1,x_i)$ contradicts Lemma \ref{lem:quad}
because all edges from the inner terminals to 1 have color 1 and the edges to $x_i$
must have color 2 (because they intersect $(1,x_{i+1}), (2,x_{i+1})$).

{\em Case 2: $(x_i,x_{i+1})$ has color 2 or 3}.
Assume wlog that $(x_i,x_{i+1})$ has color 2 (the case of color 3 is symmetric).
Consider the triangles $(1,x_i,x_{i+1})$, $(2,x_i,x_{i+1})$ 
and the quads attached to their edges.
If one of the inner terminals is right of $x_{i+1}$ then all of them must be there
(within the prime region),
because the edge $(1,x_{i+1})$, the path $(2,x_i,x_{i+1})$
and the edge $(2,x_{i+1})$ use all 3 colors, 
hence there cannot be an edge from a node right of $x_{i+1}$ to a node
(other than $x_i$) left of $x_{i+1}$.
Node $x_i$ can reach nodes right of $x_{i+1}$ only with color 2.
By Lemma \ref{lem:triangle}, the inner terminals of the quads of the
triangles $(1,x_i,x_{i+1})$, $(2,x_i,x_{i+1})$ cannot be right of $x_{i+1}$,
hence they are all left of  $x_{i+1}$.

If there is an edge from $x_{i+1}$ to a node left of $x_i$, then the
edge must be colored 3 (because of the edges $(1,x_i), (2,x_i)$), 
in which case node 2 cannot reach the interval $(x_i,x_{i+1})$
(because the edge $(x_i,x_{i+1})$ was assumed to have color 2).
Then all inner terminals of the quad for the edge $(2,x_{i+1})$ must be
left of $x_i$ and their edges to 2 and $x_{i+1}$ are colored 2 or 3, 
contradicting Lemma \ref{lem:quad}.

Therefore, there is no edge from $x_{i+1}$ to a node
left of $x_i$,
hence all inner terminals of the quads for the edges $(1,x_{i+1})$ and $(2,x_{i+1})$
are in the interval $(x_i,x_{i+1})$. All edges from this interval
to 1 have color 1 and to 2 have color 3 (because of the edge $(x_i,x_{i+1})$).
Suppose that the closest inner terminal to $x_{i+1}$ belongs to the $(1,x_{i+1})$ quad,
then its edge to 1 has color 1.
Consider the quad of the edge $(2,x_{i+1})$: the edges from the inner terminals
to 2 and $x_{i+1}$ have color 2 or 3, contradicting Lemma \ref{lem:quad}.
Similarly, if the closest inner terminal to $x_{i+1}$ belongs to the $(2,x_{i+1})$ quad,
then its edge to 2 has color 3, hence the edges from the inner terminals
of the quad of the edge $(1,x_{i+1})$ to 1 and $x_{i+1}$ have colors 1 and 2,
contradicting again Lemma \ref{lem:quad}.
\end{proof}

Let $x_i, i \notin \{ 1, n-1, n \}$ be a vertical node in the prime region 
that is at stretch distance more than 80 from the endpoints
of the prime region.
Then both
its successor $x_{i+1}$ on the path $p$ and its predecessor $x_{i-1}$
are also in the prime region, at stretch distance more than 60
from its endpoints. Therefore, 
the edges from node 1 to $x_{i-1}, x_i, x_{i+1}$
have all color 1, and the edges from node 2 to
$x_{i-1}, x_i, x_{i+1}$ have all the same color (2 or 3) by Lemma \ref{lem:same}.
Assume without loss of generality that $x_i$ is left of $x_{i+1}$,
and that the edges from 2 to $x_{i-1}, x_i, x_{i+1}$ have color 2.
Using the same argument as in Claim \ref{cl:next1} (in the proof of Lemma \ref{lem:stretch}),
we can deduce that $x_{i-1}$ is not embedded between $x_i$ and $x_{i+1}$.
(If $x_{i-1}$ is in $(x_i, x_{i+1})$ then the edge $(x_i,x_{i+1})$ must have color 3 (because $(1,x_{i-1})$ has color 1 and $(2,x_{i-1})$ has color 2), 
hence nodes 1 and 2 can reach the interval $(x_i, x_{i+1})$
only with colors 1 and 2 respectively. The argument in the
last paragraph of the proof of Claim \ref{cl:next1} applies then verbatim.)
Similarly, $x_{i+1}$ is not between $x_i$ and $x_{i-1}$.
Therefore, $x_{i-1}$ is left of $x_i$.
Similarly, since $x_{i+1}$ is at stretch distance more than 60
from the endpoints of the prime region,
its successor $x_{i+2}$ is also in the prime region, and it is right of $x_{i+1}$. 
All edges from $x_{i-1}, x_i, x_{i+1}, x_{i+2}$ to terminal 1 are colored 1, 
and all their edges to terminal 2 have the same color by Lemma \ref{lem:same}, say color 2.

\begin{lemma}\label{lem:fin1} 
Let $x_i, i \notin \{ 1, n-1, n \}$ be any vertical node in the prime region that is at stretch distance more than 80 from the endpoints
of the prime region.
There is an index $j \in \{ i-1, i, i+1 \}$,
such that the edges $(x_j,x_{j+1}), (2,x_j), (2,x_{j+1})$ all
have the same color, say 2, and the centers of both triangles
$(1,x_j, x_{j+1})$, $(2,x_j, x_{j+1})$ are in the
interval $(x_j, x_{j+1})$.
\end{lemma}
\begin{proof}
As we observed before the lemma, the nodes
$x_{i-1}, x_i, x_{i+1}, x_{i+2}$ are in the prime region in this
order (or the reverse),
all edges from $x_{i-1}, x_i, x_{i+1}, x_{i+2}$ to terminal 1 
are colored 1, 
and all their edges to terminal 2 have the same color by Lemma \ref{lem:same}, say color 2.

Suppose that both centers $c_{1i}, c_{2i}$
of the triangles $(1,x_i, x_{i+1})$, $(2,x_i, x_{i+1})$ are in the
interval $(x_i, x_{i+1})$. 
Consider the copy of $Q_1$ with outer terminals 1, 2 and inner terminals
$x_i, x_{i+1}$.
Node 1 can reach the interval $(x_i, x_{i+1})$ only with color 1.
If $(x_i, x_{i+1})$ is colored 3, then node 2 can reach the interval
only with color 2, contradicting Lemma \ref{lem:inside}.
Therefore, $(x_i, x_{i+1})$ must be colored 2.
Thus, the claim holds for $j=i$.

On the other hand, suppose that a center $w$ of one of the
triangles $(1,x_i, x_{i+1})$, $(2,x_i, x_{i+1})$ is outside the
interval $(x_i, x_{i+1})$. 
Then by (the proof of) Claim \ref{cl:next2},
$w$ must be either in the interval
$(x_{i-1},x_i)$ or in the interval $(x_{i+1}, x_{i+2})$.
Suppose wlog that $w$ is in $(x_{i+1}, x_{i+2})$ - see Fig. \ref{fig:next2}; $z=x_{i+2}$ in the figure.
The edge $(x_i,w)$ has color 3 because it intersects
the edges $(1,x_{i+1})$ and $(2,x_{i+1})$.
Node $x_{i+1}$ cannot reach any node, other than 1,2, outside the
interval $(x_i, x_{i+2})$ because of the
color-1 path $(x_i,1,x_{i+2})$, the color-2 path $(x_i,2,x_{i+2})$, and
the color-3 edge $(x_i,w)$. Similarly, node $x_{i+2}$ cannot reach any node in the interval $(x_i,x_{i+1})$
because of the color-1 path $(x_i,1,x_{i+1})$, the color-2 path $(x_i,2,x_{i+1})$, and
the color-3 edge $(x_i,w)$.
Therefore, the centers $c_{1,i+1}, c_{2,i+1}$ of both triangles
$(1,x_{i+1}, x_{i+2})$, $(2,x_{i+1}, x_{i+2})$ must be in the
interval $(x_{i+1}, x_{i+2})$. The edge $(1,c_{1,i+1})$ is colored 1
and the edge $(x_i,w)$ is colored 3, therefore the edge
$(x_{i+1}, x_{i+2})$ must be colored 2.
Thus, the claim holds for $j=i+1$.
\end{proof}

Since the edge $(x_j,x_{j+1})$ has color 2,
all edges from node 2 to the open interval $(x_j,x_{j+1})$ must have color 3.
We finish the proof by showing that it is impossible to
embed and color the edges of the triangles $(1,x_j, x_{j+1})$, $(2,x_j, x_{j+1})$
so that both centers are in $(x_j,x_{j+1})$
as required by Lemma \ref{lem:fin1}.
This statement is similar to Lemma \ref{lem:inside} for the copy of $Q_1$
with outer terminals 1, 2, and inner terminals $x_i, x_{i+1}$, 
but the important  difference to Lemma \ref{lem:inside},
is that here the color available from node 2 to the open interval
$(x_j, x_{j+1})$ is different than the color of the edges from 2
to the nodes $x_j, x_{j+1}$.
On the other hand, the graph here is more involved and has
quads attached to the edges.

\begin{lemma}\label{lem:fin2}
There is no 3-page embedding in which nodes $x_j, x_{j+1}$ are in the
prime region, with color-1 edges to terminal 1, color-2 edges to terminal 2,
the edge $(x_j, x_{j+1})$ has color 2, 
all edges from the open interval $(x_j, x_{j+1})$ to 1 have color 1
and to node 2 have color 3, and the centers 
of both triangles $(1,x_j, x_{j+1})$, $(2,x_j, x_{j+1})$ 
are in the interval $(x_j, x_{j+1})$.
\end{lemma}
\begin{proof}
Assume without loss of generality that the
center $c_{1j}$ of the triangle $(1,x_j, x_{j+1})$ is left of the
center $c_{2j}$ of the triangle $(2,x_j, x_{j+1})$ -
 see Figure \ref{fig:fin2}.

\begin{figure}[h]
\centering
\vspace*{-1.1cm}
\includegraphics[scale=0.6]{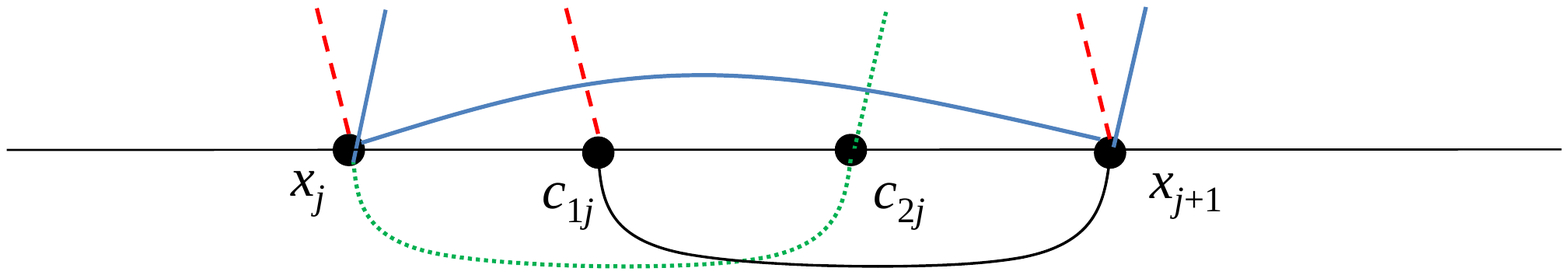}
\vspace*{-8.2cm}
\caption{}
\vspace*{-0.3cm}
\label{fig:fin2}
\end{figure}

If $x_j$ has a color-2 edge to a node in the interval $(c_{1j},x_{j+1})$
(for example, to $c_{2j}$) then the edge $(c_{1j},x_{j+1})$ must have color 1
(because of the color-3 edge $(2,c_{2j})$).
Then there is no position for the center of the triangle
$(1,x_{j+1},c_{1j})$: The center must be in the prime region
because it is adjacent to $x_{j+1}$. Node $x_{j+1}$ cannot reach left of $c_{1j}$
(because of the color-1 edge $(1,c_{1j})$, the color-2 edge from $x_j$ to 
the interval $(c_{1j},x_{j+1})$ and the edge $(2,x_j)$,
and the color-3 edge $(2,c_{2j})$);
node $c_{1j}$ cannot reach right of $x_{j+1}$ 
(because of the color-1 edge $(1,x_{j+1})$,
the color-2 edge $(2,x_{j+1})$ and the color-3 edge $(2,c_{2j})$); and
1 cannot reach the interval $(c_{1j},x_{j+1})$ (because of the color-1
edge $(c_{1j},x_{j+1})$).
Therefore, there is no color-2 edge from $x_j$ to the interval $(c_{1j},x_{j+1})$.
In particular, the edge $(x_j,c_{2j})$ must be colored 3
(it cannot be colored 1 because of the edge $(1,c_{1j})$).

Now consider the triangle $(2,x_j,c_{2j})$ and the quads attached to its edges.
Since there is no color-2 edge from $x_j$ to the interval $(c_{1j},x_{j+1})$,
there is no edge (of any color) from $x_j$ to the right of $c_{2j}$
because of the color-1 edge $(1,c_{1j})$, the color-2 edge $(2,x_{j+1})$
and the color-3 edge $(2,c_{2j})$.
Node 2 cannot reach the interval $(x_j,c_{2j})$ because 
of the color-3 edge $(x_j,c_{2j})$.
Therefore, the center of the triangle $(2,x_j,c_{2j})$
(and the inner terminals of the quad for the edge $(2,x_j)$)
must be left of $x_j$.
Note that no edge can connect a node left of $x_j$ to a node right of $x_j$
(within the prime region), other than $c_{2j}$, 
because of the color-1 edge $(1,x_j)$, the color-2 edge $(2,x_j)$,
and the color-3 path $(x_j,c_{2j},2)$.
Therefore all the inner terminals of the quads attached to all the
edges of the triangle $(2,x_j,c_{2j})$ are left of $x_j$
(and within the prime region).
Their edges to $c_{2j}$ must be colored 3 because of the
conflicting edges $(1,x_j), (2,x_j)$. 
This contradicts Lemma \ref{lem:triangle}.

It follows that there is no 3-page embedding satisfying the conditions of the lemma.
\end{proof}

Lemmas \ref{lem:fin1} and \ref{lem:fin2} provide the desired contradiction to the assumption that $G$ can be embedded in three pages.
 We conclude:

\begin{theorem}\label{thm:main}
There is no 3-page embedding of the graph $G$.
\end{theorem}

\bigskip
\noindent{\bf Acknowledgment.} Work supported by NSF Grants CCF-1703925, CCF-1763970. We thank the anonymous referees for their helpful comments.

\end{document}